\documentclass[12pt]{article}

\usepackage{amssymb, amsmath, amsthm, graphicx, verbatim,float}
\usepackage[left=1in,top=1in,right=1in]{geometry}

      \theoremstyle{plain}
      \newtheorem{theorem}{Theorem}[section]
      \newtheorem{lemma}[theorem]{Lemma}
      \newtheorem{corollary}[theorem]{Corollary}
      \newtheorem{proposition}[theorem]{Proposition}
      
      \theoremstyle{definition}
      \newtheorem{definition}[theorem]{Definition}
      
      \theoremstyle{remark}
      \newtheorem{remark}[theorem]{Remark}

	\newcommand{\ZZ}{{\mathbb Z}}
	\newcommand{\RR}{{\mathbb R}}

\begin{document}

\title{Uniqueness of Ground States for Short-Range Spin Glasses in the Half-Plane}

\author{Louis-Pierre Arguin\footnote{Courant Institute of Mathematical Sciences,
New York University, New York, NY}
\and
Michael Damron\footnote{Mathematics Department, Princeton University, Princeton, NJ}
\and
C.M. Newman\footnotemark[1]
%\footnote{Courant Institute of Mathematical Sciences,
%New York University, New York, NY}
\and
D.L. Stein\footnote{Courant Institute of Mathematical Sciences and Physics Department,
New York University, New York, NY}
}
\date{November 2009}
\maketitle


\begin{abstract}
We consider the Edwards-Anderson Ising spin glass model on the half-plane
$\ZZ \times \ZZ^+$ with zero external field and a wide range of choices,
including mean zero Gaussian, for the common distribution of the collection
$J$ of i.i.d.~nearest neighbor couplings. The infinite-volume joint
distribution $\mathcal{K}(J,\alpha)$ of couplings~$J$ and ground state pairs~$\alpha$
with periodic (respectively, free) boundary conditions in the horizontal
(respectively, vertical) coordinate is shown to exist without need for subsequence
limits. Our main result is that for almost every $J$, the conditional distribution
$\mathcal{K}(\alpha~|~J)$ is supported on a {\it single\/} ground state pair. 
\end{abstract}

\section{Introduction}
\label{sec:intro}

\subsection{Background}
\label{subsec:background}
The problem of determining the number of distinct ground state pairs in
realistic spin glass models remains of primary importance to understanding
the nature of spin glasses~\cite{BY86,NS03}.  
It is a measure of the
difficulty of the problem that, despite decades of effort, it remains unresolved on a mathematically rigorous or even an
analytically but non-rigorous level for any
nontrivial dimension.

Of central interest is the Edwards-Anderson~(EA) Ising model~\cite{EA75} on
${\ZZ}^d$.  The model is defined by the Hamiltonian 
\begin{equation}
\label{eq:EA}
{\cal H}_J(\sigma)= -\sum_{\langle x,y\rangle} J_{xy} \sigma_x 
\sigma_y \ ,
\end{equation}
where $J$ denotes a specific realization of the couplings $J_{xy} =
J_{\langle x,y\rangle}$, the spins $\sigma_x=\pm 1$ and the sum is over
nearest-neighbor pairs $\langle x,y\rangle$ only, with the sites $x,y$ on
the cubic lattice ${\ZZ}^d$.  The $J_{xy}$'s are independently chosen from a
symmetric, continuous distribution with unbounded support, such as Gaussian
with mean zero.

Of course, for $d=1$, the multiplicity of infinite-volume ground states is
exactly two --- i.e., a single ground state pair (GSP) of spin
configurations related to each other by a global spin flip. At
the opposite extreme, the infinite-range
Sherrington-Kirkpatrick model~\cite{SK75}, which is expected to possess a
similar themodynamic structure to the EA model as $d \to \infty$, is known to
have (in an appropriate sense) 
an infinite number of GSP's~\cite{BY86,MPV87}. But for any
nontrivial dimension --- i.e., $2\le d<\infty$ --- there are very few
analytical results.  One exception is the {\it highly disordered
model\/}~\cite{NS94,NS96a}, in which a transition in the number of GSP's
from one at low dimensions to infinitely many at high dimensions is known
(although only partially proved),
with the probable crossover dimension predicted either as $d=8$~\cite{NS94} or $d=6$~\cite{JS09}. However, that model has an unusual (volume-dependent) coupling distribution, and 
should not be considered
``realistic.''

There have been efforts to solve the problem in two dimensions, whose
special properties and simplifications might lend itself more readily to
analytical approaches.  Although over the past decade numerical
simulations~\cite{Middleton99,PY99} have pointed toward a single pair of
ground states in the EA model in two dimensions, mathematically the problem
remains open, and the issue is not completely settled~\cite{Loebl04}. A
partial result due to Newman and Stein~\cite{NS00,NS01}, which we will make use
of here, supports the conjecture of a single GSP for $d=2$, but is not
inconsistent with many GSP's.

In this paper, we provide the first rigorous result for an EA model in
nontrivial dimension -- on the half-plane.  We begin with
finite-volume measures corresponding to joint distributions of the
couplings and ground states, and prove that these converge to a
unique limit.  More significantly, we also 
prove that the conditional distribution
of the limiting measure 
(for almost every coupling realization) is supported on a single GSP.
A technical tool that we will use to obtain these results is the
{\it metastate\/} \cite{AW90,NS96b}, which will be defined, reviewed and extended in
Section~\ref{sec:metastates}.

\subsection{Preliminaries}
\label{subsec:preliminaries}
We hereafter restrict ourselves to the EA spin glass model on the
half-plane.  As in the general case, we assign i.i.d.~random variables (the
{\it couplings}, with product measure $\nu$) to the
nearest-neighbor edges of the upper half-plane, whose vertex set is
\[ H = \{ (m,m')~:~m,m' \in \ZZ \mbox{ and } m' \geq 0\} \]
and whose edge set we denote by $E$.  Throughout the paper, we define the dual upper half-plane to be the graph with vertices
\[ H^* = \{(x+1/2,y-1/2)~:~(x,y) \in H\} \]
(the {\it dual vertices}) and with all nearest-neighbor edges 
(the {\it dual edges}), except for those between dual vertices in the 
{\it dual x-axis} 
\[ X^* = \{(x+1/2,-1/2)~:~ x \in \ZZ\} \ . \]

For each $n$, we consider the box
$\Lambda_n := [-n,n] \times [0,2n]$ and define the (random) energy on spin
configurations $\sigma \in \{-1,+1\}^{\Lambda_n}$ by
\begin{equation}\label{eq:energydef}
{\cal H}_n(\sigma) = - \sum_{\langle x,y\rangle \in E_n} J_{xy}\sigma_x\sigma_y \ ,
\end{equation}
where the sum is over all nearest-neighbor edges $E_n$ with at least one endpoint
in the box $\Lambda_n$. 
(We remark that we could as well consider rectangular boxes
$[-n,n] \times [0,n']$ and let $n,n' \to \infty$ independently
of each other, with essentially no changes to our arguments.)
We use periodic boundary conditions on the left and
right sides of the box, and free boundary conditions on the top and
bottom. In this case, the energy \eqref{eq:energydef} has the symmetry
$\sigma\mapsto -\sigma$. The relevant space is then 
$\{-1,+1\}^{\Lambda_n}$ modulo a global spinflip and we denote it by
$\tilde{\Sigma}_{\Lambda_n}$. More generally, we will write
$\tilde{\Sigma}_{A}$ for the set of configurations $\{-1,+1\}^A$ modulo spinflip for a
subset $A$ of $\ZZ^2$. 
Because the coupling
distribution is continuous, there exists (with probability one) a unique element of
$\tilde{\Sigma}_{\Lambda_n}$ with lowest energy.  For
each $n$, denote by $\alpha_{n,J}$ the pair of spin configurations of least
energy, i.e., the {\it ground state pair} (GSP) in $\Lambda_n$.

The results in this paper concern limits of the GSP's $\alpha_{n,J}$.  For
this, we need a definition of a GSP in infinite volume.  We will say that
$\alpha \in \tilde{\Sigma}_{H}$ is an {\it infinite-volume GSP} if 
for each dual circuit $C^*$ in $H^*$ and for each path $P^*$ in $H^*$ 
which begins and ends in distinct dual vertices of the dual $x$-axis 
$X^*$, we have
\begin{equation}\label{eq:gsprop2}
\sum_{\langle x,y \rangle \in C^*} J_{xy}\alpha_x\alpha_y > 0 
\mbox{ and } \sum_{\langle x,y \rangle \in P^*} J_{xy}\alpha_x\alpha_y > 0 \ .
\end{equation}
It is easy to see that with our choice of boundary conditions, this is equivalent to the more familiar characterization that $\alpha$ is an infinite-volume GSP if 
and only if
\begin{equation}\label{eq:gsprop}
\sum_{\langle x,y \rangle \in \partial S} J_{xy} \alpha_x\alpha_y > 0 \text { \ for every finite set of vertices $S$}\ .
\end{equation}
Here, $\partial S$ refers to all edges which have one endpoint in $S$ 
and one not in $S$. 

For any edge $\langle x,y \rangle$, we say that the coupling $J_{xy}$ 
is {\it satisfied} in $\alpha \in \tilde{\Sigma}_{H}$ if the 
inequality $J_{xy}\alpha_x\alpha_y > 0$ holds; otherwise we will call 
the coupling {\it unsatisfied}.
When we say that the coupling at a dual 
edge in $H^*$ is satisfied (unsatisfied), we mean that the coupling at 
the edge in the original lattice $H$ (the 
perpendicular bisector edge which is dual to this 
dual edge) is satisfied (unsatisfied).
The {\it interface} between configurations $\alpha$ and $\beta$ in  
$\tilde{\Sigma}_{H}$ is the set of dual edges (and 
their endpoints, which are dual vertices) whose couplings are satisfied
in exactly one of the two configurations.  The interface between $\alpha$ and $\beta$
will be denoted $\alpha \Delta \beta$ and we will call each connected 
component of it a {\it domain wall}.

\subsection{Main Results}\label{subsec:mainresults}

Recall the definition of the box $\Lambda_n$ from 
Section~\ref{subsec:preliminaries}.  Let $\mathcal{K}_n$ be the joint 
distribution of the couplings and the corresponding GSP on $\Lambda_n$, 
using the boundary conditions listed in Section~\ref{subsec:preliminaries} 
(periodic on the left and right sides of the box and free on the top
and bottom).  The first theorem below states that the measures 
$\mathcal{K}_n$ converge in the sense of finite-dimensional distributions.  
The second theorem states that the limiting measure $\mathcal{K}$ is supported on
a single GSP for almost every (a.e.) $J$.

\begin{theorem}\label{thm:maintheorem1}
The sequence of measures $(\mathcal{K}_n)$ converges as $n \to \infty$.
The limiting measure $\mathcal{K}$ is supported on infinite-volume GSP's;
in other words, for a.e.~coupling configuration $J$, the conditional
distribution $\mathcal{K}(~\cdot~|~ J)$ is supported on GSP's for
that $J$.
\end{theorem}

\begin{theorem}\label{thm:maintheorem2}
The limiting measure $\mathcal{K}$ has the property that for a.e.~$J$, 
the conditional distribution $\mathcal{K}(~\cdot~|~ J)$ 
is supported on only a single GSP.
\end{theorem}

Our analysis will be based on the
concept of metastate, developed in
\cite{AW90,NS96b,N97,NS97,NSBerlin,NS98}, and the {\it excitation
metastate}, introduced in \cite{NS01}. Theorem \ref{thm:maintheorem3}
below
is stated in this framework and easily yields
the above main results.

\section{Metastates and Outline of the Proof}
\label{sec:metastates}

\subsection{Metastates on Ground States}

There are different, mathematically equivalent, ways to define the
metastate; the one that will be most convenient for our purposes is the first
version, using joint distributions, due to Aizenman and
Wehr~\cite{AW90}. The metastate concept provides both an appropriate setting for infinite-volume Gibbs states 
(or in our context, ground state pairs) for disordered systems and a useful tool
for mathematical analysis.  We present the definitions for boundary conditions such as periodic or free, where ground state spin configurations are considered modulo a global spinflip,
though more general cases are similar. The treatment for $\ZZ^d$ is the
same as for $H$.

For any $\Lambda\subset H$ finite and fixed $\alpha\in\tilde{\Sigma}_\Lambda$ 
we define a probability measure $\delta_\alpha$ supported on $\alpha$: for any $\alpha'\in\tilde{\Sigma}_{\Lambda}$
$$
\delta_{\alpha}(\alpha'):= 
\prod_{\substack{\langle x,y \rangle\\ 
x,y\in \Lambda}}1_{\alpha_x\alpha_y}(\alpha'_x\alpha'_y) \ ,
$$
where $1_{\alpha_x\alpha_y}(\alpha'_x\alpha'_y)=1$ if 
$\alpha_x\alpha_y=\alpha'_x\alpha'_y$ and $0$ otherwise.

For the sequence $\Lambda_n$ with $\Lambda_n\to H$ we recall that $E_n$ 
stands for the set of edges of $\Lambda_n$ and $\alpha_{n,J}$, for the unique 
GSP on $\Lambda_n$. We consider the joint distribution on 
$\tilde{\Sigma}_{\Lambda} \times \RR^{E_n}$ of the couplings together 
with the measure supported on $\alpha_{n,J}$:
\begin{equation}\label{eq:findimkappa}
\mathcal{K}_n:= \delta_{\alpha_{n,J} }\ \nu_n(d J)\ ,
\end{equation}
where $\nu_n$ is the i.i.d.~product measure for the couplings in $\Lambda_n$.
A standard compactness argument leads to the existence along subsequences
of a limiting measure $\mathcal{K}$ in the sense of finite-dimensional
distributions. More precisely, 
for every subsequence, there exists a subsubsequence $n_k$ such that for
any finite $\Lambda\subset H$, $m < \infty$ and $A$ any measurable event in
$\tilde{\Sigma}_{\Lambda} \times \RR^{E_m}$,
$$ \mathcal{K}(A)=\lim_{k\to\infty}\mathcal{K}_{n_k}(A)\ .$$
The reader is referred to Lemma B.1 in \cite{N97} for more details.
Moreover, by construction, the conditional of $\mathcal{K}$ 
given $J$ is supported 
on GSP's for that $J$ since the property \eqref{eq:gsprop} is clearly preserved. 
Since the space  $\{-1,+1\}^H$ is Polish, the distribution conditioned
on the couplings exists for $\nu$-almost all $J$, yielding the 
following definition
of the metastate.

\begin{definition}[Metastate for Ground States]
A metastate $\mathcal{K}_J$ is
a probability measure on  $\tilde{\Sigma}_{H}$ obtained by conditioning
a limit $\mathcal{K}$ of finite-volume measures \eqref{eq:findimkappa}
on the realization $J$ of the couplings. In particular, for almost every $J$, it is supported
on GSP's for that $J$.
\end{definition}

Another construction of metastates, 
referred to as {\it empirical metastates},
consists of taking subsequential limits of the empirical measures
$ \frac{1}{N}\sum_{n=1}^N\delta_{\alpha_{n,J}}$.
It can be shown that there exist subsequences for which both constructions
agree (more accurately, $\alpha_{n,J}$ may need to be replaced by
$\alpha_{m_n,J}$ where $m_n$ is increasing with $n$;
see Appendix B of \cite{N97}). 
We can now present a precise version of our main 
result from which Theorems \ref{thm:maintheorem1} and 
\ref{thm:maintheorem2} follow.
\begin{theorem}
\label{thm:maintheorem3}
Let $\alpha$ and $\beta$ be two GSP's sampled independently from 
metastates $\mathcal{K}_J$ and $\mathcal{K}'_J$ for the same realization 
$J$ of the couplings. Then, for $\nu$-almost all $J$, $\alpha = \beta$ with $\mathcal{K}_J\times\mathcal{K}'_J$-probability one.  Thus for almost every $J$ there exists a unique metastate 
$\mathcal{K}_J$ and it is supported on a single GSP.
\end{theorem}
The theorem implies the existence of the limit $\mathcal{K}$ of 
finite-volume measures of the form \eqref{eq:findimkappa} since it shows
that every convergent subsequence has the same limit.

\subsection{Outline of Proof}
The proof of Theorem \ref{thm:maintheorem3} consists of two main
parts. First we show that if there are two distinct GSP's 
obtained from metastates
for the same
$J$, then their interface contains (with positive probability)
infinitely many domain walls.  We focus on the density properties of
{\it tethered} domain walls, those that intersect the dual $x$-axis. 
The proof of the existence of tethered domain walls requires
an extension of the metastate that includes excited states. 
These measures are discussed in Section \ref{subsec:excitation}.

Second, we construct from two (possibly different) metastates for the half-plane $H$
a measure $\mu^*$ on pairs of GSP's for the full plane.
Under the assumption that there is more than
one GSP for the half-plane metastates, we show that with positive probability, two
ground states sampled independently from $\mu^*$ have an
interface which contains at least two distinct domain walls. This would lead to a
contradiction by virtue of a theorem of \cite{NS00,NS01} which prohibits
the existence of more than one domain wall in the full plane. The theorem as stated in \cite{NS00,NS01} does not apply to the measure $\mu^*$,
since $\mu^*$ is not a product of metastates constructed from finite-volume GSP's for the full plane. 
However, we will give in Section~\ref{subsec:plane} a more general version of that theorem 
which makes evident that the measure $\mu^*$ satisfies all the 
necessary properties for the original proof of \cite{NS00,NS01} to hold.

\subsection{Excitation Metastates}\label{subsec:excitation}

The excitation metastate is a probability measure on configurations of
minimal energy where spins in some finite subset of $H$ are specified. It
includes the metastate constructed previously as a marginal. 
The measure we will use here is slightly different from the original definition in \cite{NS01}, 
although it contains essentially the same information.  
The main purpose of the excitation metastate is to express sufficient conditions for
an infinite-volume state to become a GSP when finitely many couplings are modified
as well as sufficient conditions for a GSP to lose this property
(cf. Propositions \ref{prop:crit}, \ref{prop:supersatisfy} and \ref{prop:2bond_excit}).
We will apply this framework to prove Corollary \ref{cor:dw}, a general result on ground state interfaces which proves the existence of tethered domain walls.

Let $A \subset H$ be finite and $\eta_A \in \tilde{\Sigma}_{A}$.  
We suppose that $n$ is large enough so that $A \subset \Lambda_n$.  
For a given coupling configuration $J$, we define the {\it excited state} $\alpha_{n,J}^{\eta_A}$ 
as the element of $\tilde{\Sigma}_{\Lambda_n}$ which minimizes the Hamiltonian ${\cal H}_n$ 
subject to the constraint that it equals $\eta_A$ on the set $A$.  
For any two pairs $\eta_A$ and $\eta_A'$, we define the energy difference
\begin{equation}\label{eq:energydiff}
\Delta {\cal E}_{n,J}(\eta_A, \eta_A') := {\cal H}_n (\alpha_{n,J}^{\eta_A}) - {\cal H}_n (\alpha_{n,J}^{\eta_A'})\ . 
\end{equation}
Clearly, the configuration $\eta_A$ for which $\alpha_{n,J}^{\eta_A}$ is the ground state in $\Lambda_n$ is determined
by these energy differences. More generally, for any subset $B\subset A$ and a given configuration $\eta_B$
on $B$, the state of minimal energy among the states $\alpha_{n,J}^{\eta_A}$ with configuration $\eta_B$ on $B$ is determined
by these energy differences. For example, finding the ground state corresponds to the case $B=\emptyset$. 
Similar quantities were introduced in \cite{NS01}, 
but the energy differences there were defined between an excited state and the ground state.
We now use two excited states in the definition since the energy difference has then a 
natural decomposition as we shall see.  

Let $J_A$ be $J$ restricted to couplings with both endpoints in $A$.
We denote by ${\cal H}_{J_A}(\cdot)$ the Hamiltonian like \eqref{eq:energydef} but with the sum restricted to $x,y \in A$.  
We write the energy difference for the couplings with at most one endpoint in $A$
$$
\Delta {\cal E}_{n,J}^{ext}(\eta_A, \eta_A') = \left[ {\cal H}_n (\alpha_{n,J}^{\eta_A}) - {\cal H}_{J_A}(\eta_A) \right] - \left[ {\cal H}_n (\alpha_{n,J}^{\eta_A'}) - {\cal H}_{J_A}(\eta_A') \right] \ .
$$
Letting $h(\eta_A, \eta_A', J_A) = {\cal H}_{J_A}(\eta_A) - {\cal H}_{J_A}(\eta_A')$, we can rewrite the difference \eqref{eq:energydiff} as
\begin{equation*}\label{finvolrealtion}
\Delta {\cal E}_{n,J}(\eta_A,\eta_A') = \Delta {\cal E}_{n,J}^{ext}(\eta_A,\eta_A') +  h(\eta_A, \eta_A', J_A) \ .
\end{equation*}
This decomposes the energy difference $\Delta {\cal E}_{n,J}$ into two pieces: 
the {\it exterior} energy difference and an {\it interior} term which depends only on variables inside of $A$. 
We can interpret the exterior energy difference as a sort of boundary condition.
Once this term is known for all $\eta_A$'s and $\eta_A'$'s, the energy differences are determined for every $J_A$
through the functions $h$. Hence so are all excited states and excitation energies for subsets $B\subset A$.
We now highlight four important properties following directly from the definitions. 
\begin{enumerate}
\item For any $\eta_A$, $\alpha_{n,J}^{\eta_A}$ has the GSP property \eqref{eq:gsprop} in $\Lambda_n\backslash A$.
\item For any $\eta_A,\eta_A'$ and $\eta_A''$: $\Delta {\cal E}_{n,J}^{ext}(\eta_A,\eta_A') + \Delta {\cal E}_{n,J}^{ext}(\eta_A',\eta_A'')=\Delta {\cal E}_{n,J}^{ext}(\eta_A,\eta_A'')$.
\item For any $\eta_A$ and $\eta_A'$, $\Delta {\cal E}_{n,J}^{ext}(\eta_A,\eta_A')$ and $\alpha_{n,J}^{\eta_A}$ do not depend on $J_A$.
\item For $B\subset A$ and $\eta_B\in \tilde{\Sigma}_{B}$, let $\eta_A^*$ be the (almost surely) unique element of $\tilde{\Sigma}_{A}$ such that: i) $\eta^*_A=\eta_B$ on $B$ and ii) $\Delta {\cal E}_{n,J}^{ext}(\eta_A,\eta^*_A) +  h(\eta_A,\eta^*_A ,J_A) \geq 0$ for all $\eta_A$ with $\eta_A=\eta_B$ on $B$. Also define ${\eta'}_A^{*}$ similarly with $\eta_B$ replaced by $\eta'_B$. Then,
\end{enumerate}
\begin{equation}\label{eq:prop4}
\alpha_{n,J}^{\eta_B} = \alpha_{n,J}^{\eta^*_A},~
\Delta \mathcal{E}_{n,J}^{ext}(\eta_B,\eta_B')=\Delta \mathcal{E}_{n,J}^{ext}(\eta_A^*,{\eta'}_A^{*})+ h(\eta^*_A, {\eta'}_A^{*}, J_A)-h(\eta_B, \eta_B', J_B)\ .
\end{equation}

Since the above finite-volume relations use the variables in a fixed set of vertices, 
they naturally extend to the infinite volume.  
We introduce the joint measure on excited states and excitation energies in all finite subsets:
\begin{equation}\label{eq:findimexcited}
\mathcal{K}_n^{\#}:=\Big(\prod_{\substack{A\subset\Lambda_n\\ \eta_A, \eta'_A\in\tilde{\Sigma}_{A}}}\delta_{(\alpha_{n,J}^{\eta_A},\Delta \mathcal{E}_{n,J}^{ext}(\eta_A,\eta_A'))}\ \Big)\ \nu_n(dJ)\ . 
\end{equation}
A compactness argument implies weak convergence along some subsequence
of the measures $\mathcal{K}^\#_n$ to an infinite-volume measure $\mathcal{K}^\#$ that may depend on the choice of the subsequence. This leads to the definition of an excitation metastate for almost all $J$ through conditioning.

\begin{definition}[Excitation Metastate]
An excitation metastate $\mathcal{K}^{\#}_J$  for a realization $J$ of the couplings is a
joint distribution on the collection 
$\left(\alpha_J^{\eta_A},\Delta \mathcal{E}_J^{ext}(\eta_A,\eta_A') \right)_{\substack{A\subset H \text{ finite}\\ \eta_A, \eta_A' \in\tilde{\Sigma}_{A}}}$ 
, where $\alpha_J^{\eta_A}\in\tilde{\Sigma}_{H}$ and $\Delta \mathcal{E}_J^{ext}(\eta_A,\eta_A') \in \RR$, obtained 
by conditioning on $J$ a limit $\mathcal{K}^{\#}$ of finite-volume measures of the form \eqref{eq:findimexcited}. 
The index $\#$ may be thought of as running over all the finite subsets $A$.
\end{definition}

We write ${\cal K}_J^A$ for the marginal of ${\cal K}_J^{\#}$ on exterior excitation energies and excited states for the set $A$.  
We stress that ${\cal K}_J^A$ for $A=\emptyset$ (or $A$ a singleton site) is simply the metastate on ground states.
It is easily checked that the set of measures satisfying the above four properties
is closed under taking convex combination and limits. In particular,
the infinite-volume measure ${\cal K}^\#_J$ satisfies analogous properties.
\begin{lemma}\label{lem:indep}
Let $\mathcal{K}^{\#}_J$ be an excitation metastate on $\left(\alpha_J^{\eta_A},\Delta \mathcal{E}_J^{ext}(\eta_A,\eta_A') \right)_{\substack{A\subset H \text{ finite}\\ \eta_A, \eta_A' \in\tilde{\Sigma}_{A}}}$. For any finite set $A$,
\begin{enumerate}
\item with $\mathcal{K}_J^\#$-probability one, $\alpha_J^{\eta_A}$ is a GSP on $H\backslash A$ for any $\eta_A$;
\item with $\mathcal{K}_J^\#$-probability one,
$\Delta {\cal E}_{J}^{ext}(\eta_A,\eta_A') + \Delta {\cal E}_{J}^{ext}(\eta_A',\eta_A'')=\Delta {\cal E}_{n,J}^{ext}(\eta_A,\eta_A'')$ for any $\eta_A$, $\eta_A'$ and $\eta_A''$;
\item ${\cal K}^A_J$ does not depend on $J_A$;
\item if $B\subset A$, then with $\mathcal{K}_J^\#$-probability one, the variables $(\alpha_J^{\eta_A},\Delta\mathcal{E}_J^{ext}(\eta_A,\eta_A'))_{\eta_A,\eta_A'}$
and  $(\alpha_J^{\eta_B},\Delta\mathcal{E}_J^{ext}(\eta_B,\eta_B'))_{\eta_B,\eta_B'}$ satisfy the equivalent of \eqref{eq:prop4} with the subscript
$n$ removed.
\end{enumerate}
\end{lemma}

Using the above properties of the excitation metastate, we can study the GSP appearing in the metastate
as a function of finitely many couplings.  The basic procedure we use is as follows.  
Letting $A$ be a finite subset of $H$, we first sample exterior excitation energies and excited states with ${\cal K}_J^A$. 
Property {\it 3} guarantees that this can be done independently of the couplings in $A$; therefore, we think of it as an ``exterior'' realization.
We then use property {\it 4} to determine the GSP as a function of this realization and the ``interior'' realization -- the couplings inside $A$.
We start with the case where $A=\{x,y\}$ consists of the endpoints of an edge $b=\langle x,y \rangle$; in this setting we will use $A$ and $b$ interchangeably. 
There are then two possible $\eta_b$'s (up to a joint spinflip): $(+1,+1)$ and $(-1,+1)$. 
To lighten notation we write $\eta_b=+_b$ when the spins at $x$ and $y$ have the same sign and $\eta_b=-_b$ when they are opposite. 
Keeping track of $b$ will be helpful when dealing with more than one edge.
The corresponding pairs of configurations sampled from the excitation metastate will respectively be denoted by
$\alpha_J^{+_b}$ and  $\alpha_J^{-_b}$.

\begin{proposition}\label{prop:crit}
Let $b=\langle x,y \rangle$ be an edge and $(\alpha_J^{\eta_b},\Delta \mathcal{E}_J^{ext}(\eta_b,\eta_b'))_{\eta_b, \eta_b'=\pm_b}$ be sampled from the excitation metastate $\mathcal{K}_J^b$. 
There exists $C_J^b\in\RR$, independent of $J_{b}$, such that the GSP $\alpha_J$ is $\alpha_J^{+_b}$ for $J_b>C_J^b$ and is $\alpha_J^{-_b}$ for $J_{b}<C_J^b$. 
Precisely,
\begin{equation}\label{eq:critvalue}
C_J^b:= \frac{1}{2} \Delta \mathcal{E}_{J}^{ext}(+_b,-_b) \ . 
\end{equation}
\end{proposition}

\begin{proof}
From Lemma~\ref{lem:indep}, $C_J^b$ is independent of $J_b$.  The same lemma allows us to determine $\alpha_J$ by taking $B = \varnothing$ and $A = \{x,y\}$ in \eqref{eq:prop4}.  The values of $h(+_b,+_b,J_b)$ and $h(-_b,-_b,J_b)$ are both 0 and we have $h(+_b,-_b,J_b) = -h(-_b,+_b,J_b) = -2J_b$.  This implies that $\eta_A^* = +_b$ for $J_b >  \frac{1}{2} \Delta \mathcal{E}_{J}^{ext}(+_b,-_b)$, and $\eta_A^* = -_b$ for $J_b <  \frac{1}{2} \Delta \mathcal{E}_{J}^{ext}(+_b,-_b)$, implying \eqref{eq:critvalue}.
\end{proof}

The statement of the proposition is illustrated in Figure \ref{fig:1bond_excit}. 
We say that $C_J^b$ is the {\it critical value} of the edge $b$. 
The {\it critical contour} is the set of dual edges $\alpha_J^{+_b}\Delta\alpha_J^{-_b}$ in $H^*$. 
This contour always goes through $b$ and might be infinite. 
When $J_{b}$ crosses $C_J^b$ from above (resp. below) and the GSP $\alpha_J$ changes, 
we say that it {\it flips} from $\alpha_J^{+_b}$ to $\alpha_J^{-_b}$ (resp. from $\alpha_J^{-_b}$ to $\alpha_J^{+_b}$). 
We stress that even though the GSP flips when the critical value is crossed,
 the former minimizing GSP can possibly retain the ground state property \eqref{eq:gsprop}. 
Hence $\alpha_J^{+_b}$ and $\alpha_J^{-_b}$ might simultaneously be GSP's for $J$. 
The following gives sufficient conditions on $J_b$ for only one of them to be a GSP.

\begin{proposition}\label{prop:supersatisfy}
If 
\begin{equation}\label{eq:supersatisfy}
|J_{xy}|>\min \left\{\sum_{z:\langle z, x\rangle\in E,z \neq y}|J_{xz}|, \sum_{z:\langle z,y \rangle \in E, z \neq x}|J_{yz}|\right\}\ , 
\end{equation}
then $\alpha_x\alpha_y=sgn \ J_{xy}$ for any GSP $\alpha$. 
In particular, in the notation of Proposition \ref{prop:crit}, exactly one of $\alpha_J^{+_b}$ and $\alpha_J^{-_b}$ is a GSP when $J_{b}$ satisfies \eqref{eq:supersatisfy}.
\end{proposition}

\begin{proof}
Whenever $J_{xy}$ satisfies the inequality, we must have $\alpha_x\alpha_y=sgn \ J_{xy}$, otherwise \eqref{eq:gsprop} is violated either for $S=\{x\}$ or for $S=\{y\}$.
\end{proof}

When the condition \eqref{eq:supersatisfy} holds, we say that the coupling $J_{xy}$ is {\it super-satisfied}.  
The reader can verify that a dual edge whose coupling is super-satisfied cannot be in an interface between two GSP's.  
Furthermore, such an edge cannot be in the critical contour of another edge.

\begin{figure}[H]
\begin{center}
\includegraphics*[viewport= 1in 8.8in 12in 9.5in]{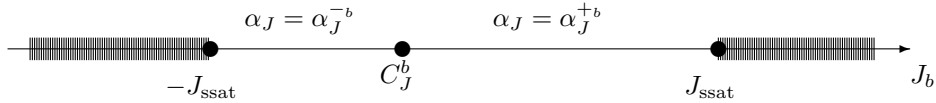}
\end{center}
\caption{The GSP $\alpha_J$ from the set $\{\alpha_J^{+_b},\alpha_J^{-_b}\}$ as a function of $J_b$. 
$J_{\text{ssat}}$ is the right-hand side of \eqref{eq:supersatisfy}. The shaded region is where $J_b$ is super-satisfied.}
\label{fig:1bond_excit}
\end{figure}

We shall later need the analogue of Proposition \ref{prop:crit} for modifications of two couplings to prove the existence of tethered domain walls.
Let $b$ and $e$ be two edges in $H$. We take $A$ to be the set of vertices which are endpoints of $b$ or $e$. 
We write $\eta_b$ or $\eta_e$ for the spin configuration up to a joint spinflip at the endpoints of the edge. Using the notation introduced in the one edge case, we have $\eta_b=\pm_b$, $\eta_e=\pm_e$.
The excitation energies for two different configurations $\eta_b,\eta_e$ and $\eta'_b,\eta'_e$ on $b$ and $e$ then reads
$\Delta\mathcal{E}^{ext}_J(\eta_b,\eta_e;\eta'_b,\eta'_e)$.
We set for convenience $C_1:=\frac{1}{2}\Delta\mathcal{E}^{ext}_J(+_b,+_e;-_b,+_e)$, $C_2:=\frac{1}{2}\Delta\mathcal{E}^{ext}_J(+_b,-_e;-_b,-_e)$, $C_3:=\frac{1}{2}\Delta\mathcal{E}^{ext}_J(+_b,+_e;+_b,-_e)$ and $C_4:=\frac{1}{2}\Delta\mathcal{E}^{ext}_J(-_b,+_e;-_b,-_e)$. Throughout the paper, it is understood that the $C_i$'s depend on $J$. By items 2 and 3 in Lemma~\ref{lem:indep} it is easily checked that $C_1-C_2=C_3-C_4$ and that the $C_i$'s are independent
of $J_b$ and $J_e$.

Moreover, taking $A$ to be the endpoints of $b$ or $e$ and $B$ to be the endpoints of just $e$, we have, again by Lemma~\ref{lem:indep}, that $\alpha_J^{+_e}$ 
is chosen from $\alpha_J^{+_b,+_e}$ and $\alpha_J^{-_b,+_e}$.
Applying the same argument as in the proof of Proposition \ref{prop:crit}, 
it follows that $\alpha_J^{+_e}$ is $\alpha_J^{+_b,+_e}$ if $J_b>C_1$ and is $\alpha_J^{-_b,+_e}$ if $J_b<C_1$.
Similarly, $C_2$ (resp. $C_3$, $C_4$) is the critical value for the state $\alpha_J^{-_e}$ (resp. $\alpha_J^{+_b}$, $\alpha_J^{-_b}$). 
We can now use these four values to describe the GSP $\alpha_J$ as a function of $J_b$ and $J_e$.

\begin{proposition}\label{prop:2bond_excit}
Let $b$ and $e$ be edges of $H$ and let $(\alpha_J^{\eta_b,\eta_e},\Delta \mathcal{E}^{ext}_J(\eta_b,\eta_e;\eta'_b,\eta'_e))_{\eta_b, \eta_b'=\pm_b; \eta_e, \eta_e'=\pm_e}$ 
be sampled from the excitation metastate $\mathcal{K}_J^\#$. 
There exists a critical set $C_J^{b,e}\subset \RR^2$ independent of $J_b$ and $J_e$
such that the GSP $\alpha_J$ is constant for $(J_b,J_e)$ in each of the four connected components of the complement of $C_J^{b,e}$
(see Figure \ref{fig:2bond_excit}).

Moreover, $C_J^{b,e}$ is the union of straight lines and is of three types determined as follows:
\begin{itemize}
\item If $C_1=C_2$, then $C_J^{b,e}$ is the union of the two lines $\{(C_1,J_e): J_e\in\RR \}$ and $\{(J_b,C_3): J_b\in\RR \}$.

\item If $C_1>C_2$, then $C_J^{b,e}$ is the union of the four rays 
$\{(J_b,C_3): J_b>C_1\}$, $\{(J_b,C_4): J_b<C_2\}$, $\{(C_1,J_e): J_e>C_3\}$ and $\{(C_2,J_e): J_e<C_4\}$ 
and the line segment $ \{(J_b,J_e): J_b-J_e=C_1-C_3 \text{ for }  C_2<J_b<C_1, ~C_4<J_e<C_3\}$.

\item If $C_1<C_2$, then $C_J^{b,e}$ is the union of the four rays 
$\{(J_b,C_3): J_b>C_2\}$, $\{(J_b,C_4): J_b<C_1\}$, $\{(C_1,J_e): J_e>C_4\}$ and $\{(C_2,J_e): J_e<C_3\}$ 
and the line segment $\{(J_b,J_e): J_b+J_e=C_1+C_4 \text{ for }   C_1<J_b<C_2, ~C_3<J_e<C_4\}$.
\end{itemize}
\end{proposition}

\begin{remark}\label{rem:contour}
The result implies that the region where $J_b$ is between $C_1$ and $C_2$ corresponds exactly to the values of $J_b$ for which $b$ is in the critical contour of $e$, 
and vice-versa for the region where $J_e$ is between $C_3$ and $C_4$. 
In particular, the middle square in the diagrams of Figure \ref{fig:2bond_excit} when $C_1\neq C_2$ gives the values of $(J_b,J_e)$ for which $b$ and $e$ share the same critical contour. 
In all three cases, $C_J^{b,e}$ is the union of the two critical lines given by the graphs of $C_J^b$ as a function of $J_e$ and $C_J^e$ as a function of $J_b$.
\end{remark}

\begin{figure}[H]
\begin{center}
\includegraphics*[viewport= 1in 6in 11in 8in]{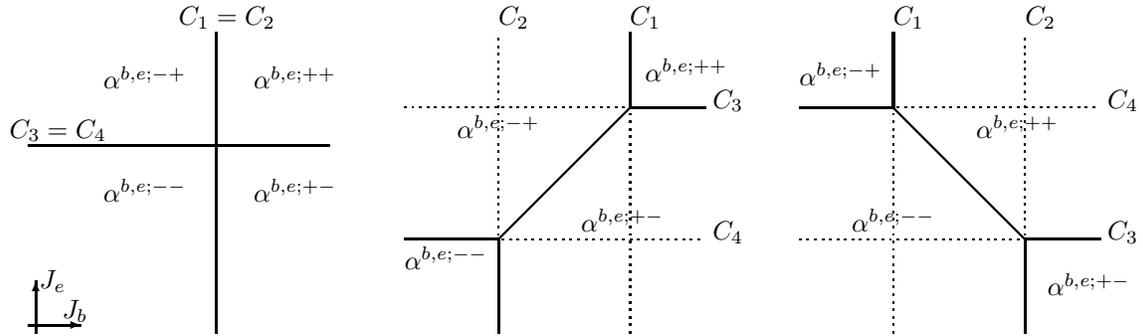}
\end{center}
\caption{The GSP $\alpha_J$ (from the set $\{\alpha_J^{\eta_b,\eta_e}~:~\eta_b=\pm1,\eta_e=\pm1 \}$) as a function of $(J_b,J_e)$. The cases $C_1=C_2$, $C_1>C_2$ and $C_1<C_2$ are depicted from left to right. The thick lines form the critical set $C_J^{b,e}$.}
\label{fig:2bond_excit}
\end{figure}

\begin{proof}
We prove the statement in the case $C_1>C_2$. The other cases are done the same way.
It suffices to compute the value of $C_J^b$ as a function of $J_e$
and $C_J^e$ as a function of $J_b$.
For a fixed $J_e$, we know from Proposition \ref{prop:crit} that the GSP $\alpha_J$ 
takes the value $\alpha_J^{+_b}$ or $\alpha_J^{-_b}$ depending on $J_b$. 
Since the exact values of $J_e$ for which $\alpha_J^{\eta_b}=\alpha_J^{\eta_b,\eta_e}$
are known from the $C_i$'s computed above, $C_J^b$ can be derived explicitly. 
For instance, in the region $C_4<J_e<C_3$ we have that $\alpha_J^{-_b}=\alpha_J^{-_b,+_e}$ and $\alpha_J^{+_b}=\alpha_J^{+_b,-_e}$.
Therefore,
$$ 
\begin{aligned}
C_J^b= \frac{1}{2}\Delta \mathcal{E}_{J}^{ext}(+_b,-_b)
=\frac{1}{2} \left( \Delta \mathcal{E}_{J}^{ext}(+_b,-_e;-_b,+_e) + 2J_e \right) 
=J_e+ C_1-C_3\ ,
\end{aligned}
$$
where we used item 4 in Lemma \eqref{lem:indep} for the second equality and item 2 for the third equality.
A similar argument in the region $J_e>C_3$ and $J_e<C_4$ yields
$C_J^b=C_1$ and $C_J^b=C_2$ respectively. The picture for $C_J^{b,e}$
is completed by computing $C_J^e$ as a function of $J_b$ the same way.  
\end{proof}

The excitation metastate is useful when investigating the interface between two GSP's because questions about interfaces translate into questions about critical values. 
Let us write $\mu$ for the product measure on two excitation metastates joint with the distribution of the
couplings:
$$
\mu := (\mathcal{K}_J^{\#} \times \mathcal{K}_J^{'\#}) \nu(dJ) \ . 
$$
We use the notation $\alpha_J$ for the states sampled through $\mathcal{K}_J$ and $\beta_J$ for the states sampled through $\mathcal{K}_J'$. 
We denote by $C_J^e(\alpha)$ and $C_J^e(\beta)$ the respective critical values for an edge $e$.
To illustrate the connection between interfaces and critical values, consider the event that a fixed edge $e$ belongs to the interface of two GSP's $\alpha_J$ and $\beta_J$.  
Suppose that this occurs with positive probability when $\alpha_J$ and $\beta_J$ are sampled from the two excitation metastates.  
In other words, suppose that 
\begin{equation}\label{eq:neqcrit_values2}
\mu(e\in\alpha_J\Delta \beta_J)=\nu\left(\mathcal{K}_J^\#\times \mathcal{K}_J'^{\#}\left(e\in\alpha_J\Delta\beta_J\right)\right)>0\ .
\end{equation}
Using Proposition~\ref{prop:crit}, we can show that the above is equivalent to the statement that $C_J^e(\alpha)\neq C_J^e(\beta)$ with positive probability:
\begin{equation}\label{eq:neqcrit_values}
\mu\left(C_J^e(\alpha)\neq C_J^e(\beta)\right)>0\ .
\end{equation}
This is intuitively clear from Figure \ref{fig:1bond_excit} and can be made precise as follows. 
We write the inner probability in \eqref{eq:neqcrit_values2} by first conditioning on the critical values and on the states that the GSP's take
as $J_e$ varies, which are the same as the $\eta_e$-excited states:
\begin{equation}\label{eq:condexcit}
\mu(e\in\alpha_J\Delta \beta_J)=\mu\left(\mathcal{K}_J^\#\times \mathcal{K}_J'^{\#}\left(e\in\alpha_J\Delta\beta_J\ 
\Big| \ (\alpha_J^{\eta_e},\beta_J^{\eta_e})_{\eta_e=\pm_e}, C_J^e(\alpha),C_J^e(\beta)\right) \right). 
\end{equation}
The inner conditioning essentially pins down two specific pictures of the form in Figure~\ref{fig:1bond_excit} (one for $\alpha_J$ and one for $\beta_J$).  
This conditional probability is a priori a function of (a) the coupling configuration $J$ and (b) the choices for critical values and excited states.  
However, for almost every fixed choice of critical values and excited states, it can be viewed simply as a function of $J_e$ 
which is defined for all values of $J_e \neq C_J^e(\alpha), C_J^e(\beta)$.  
By Proposition~\ref{prop:crit}, it is equal to $1$ when $J_e$ is between $C_J^e(\alpha)$ and $C_J^e(\beta)$ and equal to $0$ otherwise.  
Also, by Lemma \ref{lem:indep}, $J_e$ is independent of the variables on which we condition.
The probability (\ref{eq:condexcit}) can thus be computed by performing the integral over $J_e$ before the integral over all other variables.  
The result will be non-zero if and only if the critical values $C_J^e(\alpha)$ and $C_J^e(\beta)$ differ with positive probability.  
This shows the equivalence of (\ref{eq:neqcrit_values2}) and (\ref{eq:neqcrit_values}).

The following corollary is another example where the connection between interfaces and critical values is fruitful in proving results about interfaces.  
It shows that a non-empty interface $\alpha_J\Delta\beta_J$ can always be modified to pass through a given edge.  
This is trivial in the plane, but not so in the half-plane where edges at different distances to the $x$-axis are not equivalent up to translation.

\begin{corollary}\label{cor:dw}
$$\mu\left(\alpha_J\Delta \beta_J\neq \emptyset\right)>0
\Longleftrightarrow
\text{ for any fixed edge $b$ in $H$, }\mu\left(b\in\alpha_J\Delta \beta_J\right) >0 \ .
$$
\end{corollary}

\begin{proof}
The backward implication is obvious, so we prove the forward one and assume that the probability on the left is positive.  
We will first show that
\begin{equation}\label{eq:critintersect}
\mu\left(\alpha_J^{+_b}\Delta\alpha_J^{-_b}\cap\alpha_J\Delta \beta_J\neq \emptyset \right)>0 \ .
\end{equation}
To do this, we consider a realization of the interface $\alpha_J \Delta \beta_J$ and 
make a coupling modification to force (\ref{eq:critintersect}) to occur.  
Denote by $\langle x,y \rangle$ the edge dual to $b$.  
Let $P$ be a path in the dual upper half-plane which connects $x$ to a dual vertex in $\alpha_J \Delta \beta_J$.  
Let $\partial P$ be the set of dual edges which are not in $P$ but which have at least one endpoint in $P$.  
We now ``super-satisfy'' (cf. Proposition~\ref{prop:supersatisfy} and the discussion immediately following it) all
dual edges in $\partial P$ which are not in the interface $\alpha_J \Delta \beta_J$ -- see Figure~\ref{fig:modification}.  
(A small amount of care is needed in the choice of $P$ to perform this procedure, namely that the set of lattice edges
dual to those in $\partial P\backslash \alpha\Delta \beta$ cannot contain a circuit.)
Precisely, notice that since these dual edges are not in the interface, we can modify their couplings one by one,
away from their respective critical values in both states, beyond the super-satisfied threshold \eqref{eq:supersatisfy}.
Consequently neither $\alpha_J$ nor $\beta_J$ will flip.  
By construction, the critical contour of $b$ in each state $\alpha_J$ and $\beta_J$ cannot 
contain any of these super-satisfied edges but it must contain $b$.  
Since the connected components of these contours which contain $b$ are either loops or doubly-infinite paths, 
this means that they must intersect the interface.   
Therefore (\ref{eq:critintersect}) holds.

\begin{figure}
\begin{center}
\includegraphics*[viewport= 1in 6in 4.8in 8.2in]{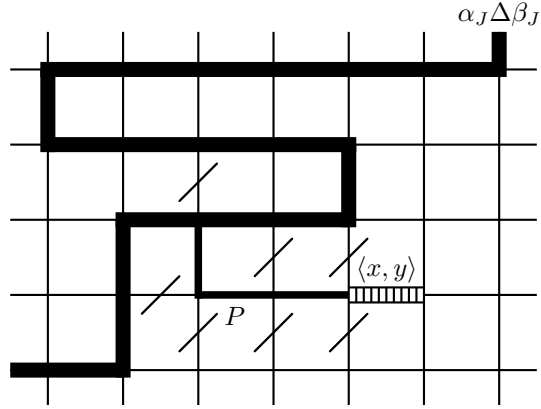}
\end{center}
\caption{An illustration of the coupling modification argument used in the proof of Corollary~\ref{cor:dw}.  All edges are in the dual lattice.  The edges crossed with diagonal line segments are super-satisfied.  Since the critical contour of $b$ (not pictured) cannot contain any crossed edges, it must intersect $\alpha_J \Delta \beta_J$.}
\label{fig:modification}
\end{figure}

From (\ref{eq:critintersect}), there exists a fixed edge $e$ such that 
\begin{equation}\label{eq:critintersect2}
\mu\left(e\in \alpha_J^{+_b}\Delta\alpha_J^{-_b}\cap\alpha_J\Delta \beta_J\right)>0 \ .
\end{equation}
If $e=b$, the corollary is proved, so assume that $e\neq b$.
Consider the excitation metastates conditioned on the critical sets $C_J^{b,e}$ as well as the $(\eta_b,\eta_e)$-excited states:
\begin{equation}\label{eq:critintersectcond}
\mathcal{K}_J^\#\times \mathcal{K}_J'^{\#}\left(e\in\alpha_J^{+_b}\Delta\alpha_J^{-_b}\cap \alpha_J\Delta\beta_J \ \Big| \ (\alpha_J^{\eta_b,\eta_e},\beta_J^{\eta_b,\eta_e})_{\eta_b=\pm_b,\eta_e=\pm_e}, C_J^{b,e}(\alpha),C_J^{b,e}(\beta)\right) \ .
\end{equation}
We will now proceed like in the proof of the equivalence of (\ref{eq:neqcrit_values2}) and (\ref{eq:neqcrit_values}).  
By Proposition \ref{prop:2bond_excit} applied both to $\alpha_J$ and to $\beta_J$ and with the help of Figure \ref{fig:2bond_excit}, 
when we fix values for the critical sets and the excited states, 
we can view this conditional probability as a function of the pair $(J_b,J_e)$ except on the critical sets $C_J^{b,e}(\alpha)$ and $C_J^{b,e}(\beta)$. 
It is equal to $1$ when $(J_b,J_e)$ is in the region where both 
i) $J_e$ is between $C_3(\alpha_J)$ and $C_4(\alpha_J)$, since for those values, $e$ is in the critical contour of $b$ for the state $\alpha_J$ (cf. Remark \ref{rem:contour}); 
and 
ii) $J_e$ is between $C_J^e(\alpha)$ and $C_J^e(\beta)$, since then $e\in\alpha_J\Delta\beta_J$.
We shall temporarily overload notation and write $C_J^b$ for either $\alpha_J$ or $\beta_J$ to denote the graph of $C_J^b$ as a function of $J_e$ and $C_J^e$ for the graph of $C_J^e$ as a function of $J_b$. 
If these two graphs coincide, we simply write $C_J^b=C_J^e$.

We now write (\ref{eq:critintersect2}) as an expectation of the conditional probability (\ref{eq:critintersectcond}).
By independence between the pair $(J_b,J_e)$ and the variables on which we condition, we may perform integration first over $J_b$ and $J_e$.
If $C_3(\alpha_J)=C_4(\alpha_J)$, the region for $(J_b,J_e)$ described by i) and ii) has zero measure so it does not contribute to the expectation \eqref{eq:critintersect2}. 
Therefore the non-trivial contribution to the expectation comes from the realizations where $C_3(\alpha_J)\neq C_4(\alpha_J)$.
We claim then that in the region for $(J_b,J_e)$ where i) is satisfied,
\begin{equation}
C_J^b(\alpha)=C_J^b(\beta)\Longrightarrow C_J^e(\alpha)=C_J^e(\beta)\ .
\label{eq:implic}
\end{equation}
To see this, note that in this region $C_J^e(\alpha)$ coincides with $C_J^b(\alpha)$ 
(they are both equal to a diagonal line segment).
Now if $C_J^b(\alpha)=C_J^b(\beta)$, 
this further implies that $C_J^b(\beta)$ contains this same diagonal line segment, and therefore $C_J^b(\beta)=C_J^e(\beta)$ in this region.
The sequence of equalities yields that $C_J^e(\alpha)=C_J^e(\beta)$ and \eqref{eq:implic} is proved.
However, the conclusion of \eqref{eq:implic} implies that 
the region described by i) and ii) has zero measure.
The full expectation would then be zero, contradicting (\ref{eq:critintersect2}).
We conclude that there must be a region of positive measure of the $(J_b,J_e)$-plane for which 
$$\mathcal{K}_J^\#\times \mathcal{K}_J'^{\#}\left( C_J^b(\alpha) \neq C_J^b(\beta)\ \Big| \ (\alpha^{\eta_b,\eta_e},\beta^{\eta_b,\eta_e})_{\eta_b=\pm_b,\eta_e=\pm_e}, C_J^{b,e}(\alpha),C_J^{b,e}(\beta)\right) $$
is non-zero. 
Integrating this first with respect to $J_b$ and $J_e$ and then with respect to the other variables, 
we see that $\mu\left(C_J^b(\alpha)\neq C_J^b(\beta) \right) > 0$,
and the corollary follows after observing the equivalence of (\ref{eq:neqcrit_values2}) and (\ref{eq:neqcrit_values}).
\end{proof}

\subsection{Uniqueness of the domain wall in the full plane}
\label{subsec:plane}
The framework of the excitation metastate developed in the last
section did not rely heavily on the choice of the finite-volume measure $\mathcal{K}_n^\#$ in equation \eqref{eq:findimexcited} or 
on the choice of underlying graph $H$.
In fact, one only needs the four properties of the excitation metastate stated in Lemma \ref{lem:indep}.
As we mentioned before, these are preserved when taking limits and convex combinations.
These properties are in particular fulfilled when one constructs the excitation metastate in the full plane $\ZZ^2$ from empirical measures as in \cite{NS00,NS01}.
The proof in \cite{NS00,NS01} that the interface between two GSP's in $\ZZ^2$ contains at most one domain wall is based on the results 
derived in the last section from these properties. 
We thus may state a more general version for use in the proof of Theorem~\ref{thm:maintheorem1}.  
\begin{theorem}[Newman-Stein]\label{thm:NSgeneral}
Let $\nu$ be the law of $J$, iid couplings on the edges of $\ZZ^2$, and let $\mu$ be a probability measure on 
$\big(J, (\alpha^{\eta_A}, \Delta\mathcal{E}^{ext}(\eta_A,\eta_A'))_{A,\eta_A,\eta_A'},(\beta^{\eta_A}, \Delta \mathcal{E}^{ext}(\eta_A,\eta_A'))_{A,\eta_A,\eta_A'}\big)$, 
where $A$ runs over all finite subsets of $\ZZ^2$ and $\eta_A,\eta_A'\in\tilde{\Sigma}_{A}$. 
Suppose that the marginal of $\mu$ on $J$ is $\nu$ and that $\mu$ is translation-invariant.
Moreover, suppose that the conditional measure of $\mu$ given $J$ on excited states $\alpha^{\eta_A}$ and $\beta^{\eta_A}$ 
and corresponding exterior excitation energies satisfies the properties of Lemma \ref{lem:indep} in $\ZZ^2$.
Then
$$  
\mu\Big(\alpha_J\Delta\beta_J=\emptyset \text{ or $\alpha_J\Delta\beta_J$ is connected}\Big)=1\ .
$$
\end{theorem}

\section{Proofs}

\subsection{Tethered domain walls}

We start with a general lemma about translation-invariant measures.
\begin{lemma}\label{lem:TIdensity}
Suppose that $\mu$ is a measure
on spin configurations in $\ZZ^2$ (or $H$) which is invariant under
horizontal translations.  
Let $A$ be an event and for any $x \in \ZZ$ let $A_x$ be the event
horizontally translated by distance $x$.  
Then with probability one, either $A_x$ does not
occur for any $x$ or $A_x$ occurs for a set of $x$ which has positive
density in $\ZZ$.
\end{lemma}
\begin{proof}
Let $\hat \mu$ be an ergodic component of the measure $\mu$.  If $\hat
\mu (A_0) > 0$ then the ergodic theorem gives that with $\hat
\mu$-probability one, the set of $x$ such that $A_x$ occurs has positive density in ${\mathbb Z}$ (in fact, it has density $\hat \mu(A_0)$).
If, on the other hand, $\hat \mu (A_0) = 0$, then with $\hat
\mu$-probability one, no $A_x$ will occur.  This shows that the lemma holds
if we replace $\mu$ by $\hat \mu$.  Since this is true for each ergodic
component, the lemma follows for $\mu$.
\end{proof}

For this section, $\alpha_J$ and $\beta_J$ will refer to two GSP's in $H$, sampled independently from
excitation metastates $\mathcal{K}_J^{\#}$ and $\mathcal{K}_J^{'\#}$ for the same coupling configuration $J$.  
Recall the definition of $\mu$ as the joint distribution of $(\alpha,\beta,J)$,
\begin{equation}\label{eq:mudefinition}
\mu = (\mathcal{K}_J^{\#} \times \mathcal{K}_J^{'\#}) \nu(dJ)\ . 
\end{equation}
Note that the measure $\mu$ is horizontally translation-invariant (because of our choice of boundary conditions in constructing $\mathcal{K}_J^{\#}$ and $\mathcal{K}_J^{'\#}$).  

\begin{definition}
A {\it tethered domain wall} is a domain wall which intersects the dual $x$-axis.
\end{definition}
\begin{proposition}\label{prop:tethdensity}
Suppose that $\mu(\alpha_J \Delta \beta_J \neq \varnothing) > 0$.  Then the interface $\alpha_J \Delta \beta_J$ contains infinitely many tethered domain walls with positive $\mu$-probability.
\end{proposition}

\begin{proof}
By Corollary~\ref{cor:dw}, with positive $\mu$-probability we can find a dual edge incident to the dual $x$-axis
 which is in $\alpha_J \Delta \beta_J$ and therefore there exists at least one
tethered domain wall. Using Lemma~\ref{lem:TIdensity} with the event
\[ A = \{\mbox{the point } (0,-1/2) \mbox{ is contained in a tethered domain wall}\} \ , \]
we see that infinitely many points of the dual $x$-axis lie in tethered domain walls.  
We claim that these points must each lie in distinct domain walls. The assertion that there are infinitely many tethered domain walls will follow. If two of these points lie in the same domain wall then there must exist a path of dual edges connecting two of them which lies entirely in $\alpha_J \Delta \beta_J$.  But this is impossible since this path must have strictly negative energy in one of $\alpha_J$ or $\beta_J$, contradicting \eqref{eq:gsprop} (it cannot be zero since the coupling distribution is assumed continuous).
\end{proof}

\bigskip
We now investigate density properties of the tethered domain walls.  For $n \geq 1$ and $k \geq 0$, define the set
\[ I_{n,k} = \{ (x,y) \in \RR^2~:~ x \in [-n,n] \mbox{ and } y = k-1/2\} \]
and let $N_{n,k}$ be the number of distinct tethered domain walls intersecting $I_{n,k}$.  Write ${\mathbb E}_\mu$ for expectation w.r.t. $\mu$.

\begin{proposition}\label{prop:tethdensityest}
The sequence $({\mathbb E}_\mu (N_{n,k}))_{n=1}^\infty$ is subadditive for fixed $k \geq 0$.  Therefore, 
\begin{equation}\label{eq:densitylimit}
\lim_{n \to \infty} {\mathbb E}_\mu (N_{n,k})/n \mbox{ exists .}
\end{equation}
Furthermore, under the assumption that $\mu(\alpha_J \Delta \beta_J \neq \varnothing) > 0$, there exists $c > 0$ such that for all $n \geq 1$ and $k \geq 0$,
\begin{equation}\label{eq:densityest3}
{\mathbb E}_\mu(N_{n,k}) \geq cn\ .
\end{equation}
\end{proposition}

\begin{proof}
Subadditivity is a straightforward consequence of translation invariance of $\mu$.  Therefore we focus on showing \eqref{eq:densityest3}.  To this end, we first show that there exists a deterministic $c_1>0$ such that with positive $\mu$-probability, the following holds for all $k \geq 0$:
\begin{equation}\label{eq:densityest}
\liminf_{n \to \infty} N_{n,k}/n > c_1\ .
\end{equation}
For the case $k =0$, let $\alpha_J \Delta \beta_J$ be a domain wall configuration which has tethered domain walls.  Using Lemma~\ref{lem:TIdensity} with $A$ as the event that the dual vertex $(1/2,-1/2)$ is in a tethered domain wall, we see that the set of $x \in \ZZ$ such that $(x+1/2,-1/2)$ is in a tethered domain wall has positive density in $\ZZ$.  Each such dual vertex is in exactly one tethered domain wall by~\eqref{eq:gsprop2}.  Therefore there exists $c_2>0$ (random) such that $\lim_{n \to \infty} N_{n,0}/n = c_2$.  By possibly decreasing $c_2$ for some configurations, we may find a deterministic $c_1>0$ such that with positive $\mu$-probability, 
\[ \liminf_{n \to \infty} N_{n,0}/n = \lim_{n \to \infty} N_{n,0}/n > c_1\ .  \]
For the rest of the proof of \eqref{eq:densityest}, restrict $\alpha_J \Delta \beta_J$ to be a domain wall configuration for which this holds.

By writing
\[ N_{n,k}/n = (N_{n,k}-N_{n,0})/n + N_{n,0}/n\ , \]
the relation (\ref{eq:densityest}) will hold in $\alpha_J \Delta \beta_J$ for all $k \geq 0$ if we show that for all $n \geq 1$,
\begin{equation}\label{eq:dwdensityrelation}
N_{n,k} - N_{n,0} \geq -2k\ .
\end{equation}
To this end, notice that $N_{n,0} - N_{n,k}$ is no bigger than the number of tethered domain walls which intersect $I_{n,0}$ but do not intersect $I_{n,k}$.  We estimate this number.  Such a tethered domain wall must originate in the set $I_{n,0}$ and cannot intersect the top side of the box $[-n,n] \times [-1,k-1]$ (i.e., the set $[-n,n] \times \{k-1\}$).  Since the domain wall must leave this box, it must leave on either the left or right side.  Therefore there exists some integer $m \in [1,k-1]$ and a dual vertex of the form $(-n-1/2,m-3/2)$ or of the form $(n+1/2,m-3/2)$ such that this vertex is in the tethered domain wall.  However, as noted before, each dual vertex can be contained in at most one tethered domain wall.  Therefore, the number of such tethered domain walls is at most $2k$.  This shows (\ref{eq:dwdensityrelation}) and completes the proof of \eqref{eq:densityest}.

We resume the proof of \eqref{eq:densityest3}.  It follows from \eqref{eq:densityest} that there is a $c_3>0$ such that for all $k \geq 0$,
\begin{equation}\label{eq:densityest2}
\lim_{n \to \infty} {\mathbb E}_\mu(N_{n,k}/n) > c_3\ .
\end{equation}
Indeed, we use Fatou's lemma and \eqref{eq:densitylimit} to see that
\[ \lim_{n \to \infty} {\mathbb E}_\mu(N_{n,k}/n) = \liminf_{n \to \infty} {\mathbb E}_\mu (N_{n,k}/n) \geq {\mathbb E}_\mu (\liminf_{n \to \infty} N_{n,k}/n)\ , \]
which is bounded away from 0 independently of $k$ by \eqref{eq:densityest}.  But now, subadditivity implies that for any $k \geq 0$, 
\[ {\mathbb E}_\mu(N_{n,k})/n \geq \inf_{n \geq 1} {\mathbb E}_\mu(N_{n,k})/n = \lim_{n \to \infty} {\mathbb E}_\mu(N_{n,k}/n) > c_3\ .\]
This proves \eqref{eq:densityest3}.

\end{proof}

\subsection{Restoring vertical translation invariance}
In the second half of the proof of Theorem \ref{thm:maintheorem3}, we create and study a measure $\mu^*$ on excited states and excitation energies on all of $\ZZ^2$.  
For this, let $\mu$ be the measure defined by (\ref{eq:mudefinition}).  For any integer $k \geq 0$ we define the shifted half-plane whose vertices are
\[ H_k = \{ (x,y-k)~:~ (x,y) \in H\} \ , \]
and whose edges are defined similarly as those of $H$ (see Section~\ref{sec:intro}).  
We define the shifted measure $\mu(k)$ simply as the push-forward measure of $\mu$ through the map which translates the origin to the vertex $(0,-k)$. 
Each $\mu(k)$ is a measure on coupling configurations for the edges of $H_k$ and corresponding pairs of excited states and
excitation energies for the finite subsets of $H_k$. We then define
\begin{equation}\label{eq:mutildedef}
\mu^*(k) = \frac{1}{k+1} \sum_{i=0}^k \mu(i) \ . 
\end{equation}
By tightness properties (see, e.g., Section~\ref{subsec:excitation} and \cite{NS01}) there is an increasing sequence $(n_k)$ such that $(\mu^*(n_k))$ converges to some measure $\mu^*$.
This is a measure on couplings for edges of $\ZZ^2$ and pairs of corresponding excitations, each of the form
$\left(\alpha_J^{\eta_A},\Delta \mathcal{E}_J^{ext}(\eta_A,\eta_A')\right)_{\substack{A\subset \ZZ^2 \text{ finite}\\ \eta_A,\eta_A'\in\tilde{\Sigma}_A}}$.
It is easy to see that it is invariant under vertical as well as horizontal translations in $\ZZ^2$.
Our aim is to use this measure to derive a contradiction from the assumption that there are multiple GSP's for the half-plane
- i.e., from $\mu(\alpha_J \Delta \beta_J \neq \varnothing) > 0$.

The proposition below says that ground state interfaces produced from $\mu^*$ would contain more than one domain wall.  
This fact along with Theorem \ref{thm:NSgeneral} will give a contradiction.  
We denote by $\alpha_J^*$ and $\beta_J^*$ the full-plane GSP's sampled from $\mu^*$.
%Denote by $\alpha_J \Delta \beta_J$ a domain wall configuration sampled from $\mu^*$ with coupling configuration $J$.

\begin{proposition}\label{prop:multipledws}
Suppose that $\mu(\alpha_J \Delta \beta_J \neq \varnothing) > 0$.  Then with positive $\mu^*$-probability, the interface $\alpha^*_J \Delta \beta^*_J$ contains at least two domain walls.
\end{proposition}

\begin{remark}
The proof below can easily be modified to show that with positive probability, $\alpha_J^* \Delta \beta_J^*$ contains infinitely many domain walls.  However, since a consequence of Theorem \ref{thm:NSgeneral} is that there is at most one domain wall, we need only show that there are at least two.
\end{remark}
\begin{proof}
Let $A_{n,k}$ be the event that at least two tethered domain walls intersect the box $[-n,n] \times [k-n,k+n]$.  
(We remark that for $n>k$ this box extends below the half-space, but that does not affect our argument.)
We first note that there exists $n_0$ such that
for $n>n_0$
\begin{equation}\label{eq:multipledws1}
\mu (A_{n,k}) > 0 \mbox{ uniformly in } k \geq 0\ .
\end{equation}
This follows from Proposition~\ref{prop:tethdensityest}.  To see this, we choose $c >0$ from \eqref{eq:densityest3} such that for all $k \geq 0$ and $n \geq 1$,
\[ {\mathbb E}_\mu(N_{n,k}) \geq cn\ . \]
We pick $n_0 = \lceil 2/c \rceil$ to give
\[ {\mathbb E}_\mu(N_{n_0,k}) \geq n_0c \geq 2\ , \]
from which (\ref{eq:multipledws1}) follows because
$$
\mu\left(A_{n_0,k}\right)\geq \mu\left(N_{n_0,k}\geq 2\right)
\geq {\mathbb E}_{\mu}\left(1_{\{N_{n_0,k}\geq 2\}}\ \frac{N_{n_0,k}-1}{2n_0-1}\right)  
%\geq {\mathbb E}_{\mu}\left(\frac{N_{n_0,k}-1}{2n_0-1}\right)
\geq \frac{1}{2n_0-1} 
\ .
$$

We now finish the proof of Proposition~\ref{prop:multipledws}.  From (\ref{eq:multipledws1}), choose $a, N > 0$ such that for all $k \geq 0$, we have $\mu(A_{N,k}) \geq a$.  For $0 < n < m$ and $k \geq 0$, define $B_{m,n,k}$ as the event that there are two dual vertices in the box $[-n,n] \times [k-n,k+n]$ which are in domain walls but such that there is no path connecting them, in the box $[-m,m] \times [k-m,k+m]$, which consists of only dual edges in domain walls.  Since tethered domain walls do not intersect, we have $B_{m,n,k} \supset A_{n,k}$, and so (\ref{eq:multipledws1}) implies that for all $M \geq N$ and for all $k \geq 0$, we have $\mu(B_{M,N,k}) \geq a.$  By construction of the measure $\mu^*$ (recall the definition as the limit of (\ref{eq:mutildedef})), we see that for any $M \geq N$,
\[ \mu^*(B_{M,N,0}) \geq a\ . \]
By definition, the events $B_{m,n,k}$ are decreasing in $m$, i.e., $B_{m+1,n,k} \subset B_{m,n,k}.$  In particular, for $n=N$, $B_{\infty,n,k} := \lim_{m \to \infty} B_{m,n,k} = \cap_{m>n}^\infty B_{m,n,k}$ satisfies
\begin{equation}\label{eq:Keq1}
\mu^*(B_{\infty,N,0}) \geq a\ . 
\end{equation}
We note that for any $n,k$,
\[ \lim_{m \to \infty} \left[ B_{m,n,k} \cap \{ \alpha^*_J \Delta \beta^*_J \mbox{ is connected} \} \right] = \varnothing \ . \]
Hence,
\begin{equation}\label{eq:Keq2}
\mu^*(B_{\infty,n,k} \cap \{ \alpha^*_J \Delta \beta^*_J \mbox{ is connected} \}) 
= \mu^* \left( \lim_{m \to \infty} (B_{m,n,k} \cap \{ \alpha^*_J \Delta \beta^*_J \mbox{ is connected}\}) \right) = 0\ . 
\end{equation}
Combining (\ref{eq:Keq1}) and (\ref{eq:Keq2}),
\[ \mu^* (\{\alpha^*_J \Delta \beta^*_J \mbox{ is not connected}\}) \geq \mu^* (B_{\infty,N,0} \cap \{ \alpha^*_J \Delta \beta^*_J \mbox{ is not connected}\}) \]
\[ = \mu^* (B_{\infty,N,0}) \geq a\ . \]
This completes the proof.
\end{proof}

We can now use properties of the measure $\mu^*$ to prove the main result, Theorem~\ref{thm:maintheorem3}.
\begin{proof}[Proof of Theorem~\ref{thm:maintheorem3}]
We first verify that the measure $\mu^*$ satisfies the hypotheses of Theorem~\ref{thm:NSgeneral}. 
By construction, it is easily seen to be translation invariant for vertical as well as for horizontal translations of $\ZZ^2$.  
The conditional measure of $\mu^*$ given $J$ is a measure on pairs of excitations
$\left(\alpha_J^{\eta_A},\Delta \mathcal{E}_J^{ext}(\eta_A,\eta_A')\right)_{A,\eta_A,\eta_A'}$ for all $A\subset \ZZ^2$ finite and $\eta_A,\eta_A'\in\tilde{\Sigma}_{A}$.
It satisfies the four properties of Lemma \ref{lem:indep}.
Indeed, these are satisfied by the conditional measure of $\mu^*(k)$ for each $k$, 
since they are satisfied by the conditional measure of $\mu$. This is so for the latter because 
it was contructed from finite-volume measures, each of which satisfied the properties, and they are carried over in 
the infinite-volume limit.  
By Theorem~\ref{thm:NSgeneral}, the measure $\mu^*$ is supported on pairs of GSP's whose interface has at most one domain wall.  
This stands in contradiction to the result of Proposition~\ref{prop:multipledws}, completing the proof. 
\end{proof}

\bigskip
{\bf Acknowledgments.}
The research reported here was supported in part by NSF grants
DMS-0604869 and OISE-0730136 and an NSF postdoctoral fellowship to M. Damron. 
L.-P. Arguin is grateful for the financial support and hospitality of Anton Bovier and the Hausdorff Center for Mathematics in Bonn during part of this work.
All the authors thank the Centre de
Recherches Math\'{e}matiques at the Universit\'{e} de Montr\'{e}al
for its hospitality during June 2009 when some of the work reported
here was done during the workshop, Disordered Systems: Spin Glasses.

\end{document}